\documentclass[a4paper, 12pt]{article}
\usepackage{amsmath,amsthm,amssymb}
\usepackage{amsfonts}
\usepackage{mathrsfs}
\usepackage[T1]{fontenc}
\usepackage{graphicx}
\usepackage{xcolor}
\usepackage[utf8]{inputenc}
\usepackage[spanish, english]{babel}
\usepackage{mathtools}
\usepackage[shortlabels]{enumitem}

\usepackage[margin=1in]{geometry}

\definecolor{enlaces}{rgb}{0.6,0.6,0.9}

\usepackage{hyperref}
\hypersetup{
    colorlinks=true,
    linkcolor=enlaces,
    filecolor=enlaces,      
    urlcolor=enlaces,
    citecolor=enlaces,
}

\usepackage[lf]{Baskervaldx} 
\usepackage[bigdelims,vvarbb]{newtxmath} 
\usepackage[cal=boondoxo]{mathalfa} 

\newtheorem{theorem}{Theorem}[section]

\newtheorem{proposition}[theorem]{Proposition}
\theoremstyle{definition}
\newtheorem{example}[theorem]{Example}
\newtheorem{remark}[theorem]{Remark}

\DeclarePairedDelimiter\modulus{\lvert}{\rvert}
\makeatletter
\let\oldmodulus\modulus
\def\modulus{\@ifstar{\oldmodulus}{\oldmodulus*}}
\makeatother

\DeclarePairedDelimiter\norm{\lVert}{\rVert}
\makeatletter
\let\oldnorm\norm
\def\norm{\@ifstar{\oldnorm}{\oldnorm*}}
\makeatother

\begin{document}
\title{Sharpness in Bohr's Inequality}
\author{Mario Guillén\thanks{Universitat Politècnica de València. cmno Vera sn 46022, València. Spain. mguisan3@posgrado.upv.es } 
\and 
Pablo Sevilla-Peris\thanks{I.U.Matemàtica Pura i Aplicada. Universitat Politècnica de València. cmno Vera sn 46022, València. Spain. psevilla@mat.upv.es\\
Research partially supported by grant PID2021-122126NB-C33 funded by \\ MICIU/AEI/10.13039/501100011033 and by ERDF/EU}}
\date{}

\maketitle

\begin{abstract}
We make a careful analysis of Bohr's inequality, in the line started by Kayumov and Ponnusamy, where some extra summand (depending on the function) is added in the right-hand side of the inequality. We analyse the inequality when  smaller radius are taken, giving sharp constants. As a result of this point of view, some previous results are improved. 
\end{abstract}

\renewcommand{\thefootnote}{\fnsymbol{footnote}} 
\renewcommand{\thefootnote}{\arabic{footnote}}

\section{Introduction}

Harald Bohr proved in \cite{bohr} the following celebrated result for holomorphic functions on $ \mathbb{D}$ (the open unit disc of the complex plane), now commonly known as \textit{Bohr's inequality}
\begin{theorem} \label{th:bohr} 
If $ f(z)= \sum_{n=0}^{\infty}a_{n}z^n$ is holomorphic and  $ |f(z)|\leq 1$ for every $ z \in \mathbb{D}$, then 
\begin{equation} \label{eq:bohr}
    \sum_{n=0}^{\infty}|a_{n}|r^n \leq 1, \, \text{ for every } r \leq \frac{1}{3}
\end{equation} 
and the radius $ 1/3$ is sharp.
\end{theorem}
This inequality has been carefully analysed and extended under different points of view. One way to look at this is to wonder if there is some room left in the left-hand side of the inequality. To put it in slightly more precise terms, if some extra term can be added, so that the inequality still holds true. A quick thought with the constant function $ 1$ immediately shows that there is no hope to find a term that is valid \textit{for every function}. However, thinking again in constant functions (this time with modulus strictly smaller than $ 1$), shows that for each particular function there is room for an extra factor. Then, the (for the moment vague) question is:  for each particular function $f$, what can be added in the left-hand side of \eqref{eq:bohr} (which necessarily has to depend on $f$) so that the inequality still holds. The first steps in this direction were given in \cite{first} and \cite{second}, where, for each $f$ and $0<r<1$, the following quantity is considered 
\begin{equation} \label{Sr}
S_{r}=S_{r}(f) \coloneqq  \int_{r\mathbb{D}}|f^{\prime}(w)|^2 \, \mathrm{d}A(w) 
= \pi \sum_{n=1}^{\infty}n \modulus{a_{n}}^{2}r^{2n} \,,
\end{equation}
(note that if $f$ is a univalent function, then this coincides with the area of $ f(r\mathbb{D})$). Then in \cite[Theorem~10]{second} the following result is given.
\begin{theorem}\label{th:second} 
 If $ f(z)= \sum_{n=0}^{\infty}a_{n}z^n$ is holomorphic and  $ |f(z)|\leq 1$ for every $ z \in \mathbb{D}$, then 
    \begin{equation*}
        \sum_{n=0}^{\infty}|a_{n}|r^n + \frac{16}{9}\left( \frac{S_{r}}{\pi-S_{r}} \right) \leq 1  \,,
    \end{equation*} 
for every $r \leq 1/3$. Moreover the constants $1/3$ and $ 16/9$ are sharp.
\end{theorem}

Here the fact that the constant $16/9$ is sharp means that it is the biggest constant so that the inequality holds for every holomorphic function on $ \mathbb{D}$ such that $ |f(z)|\leq 1$ for every $ z \in \mathbb{D}$ (the class of which we denote by $ \mathcal{B}$) and every $0 < r \leq 1/3$.  Our first aim is to give a closer look at this inequality, looking, for each fixed radius $0<R<1/3 $, for the best constant $\Lambda(R)$ so that, for each $ f(z)= \sum_{n=0}^{\infty}a_{n}z^n \in  \mathcal{B}$ we have
\begin{equation} \label{pregunta}
    \sum_{n=0}^{\infty}|a_{n}|r^n + \Lambda(R)\left( \frac{S_{r}}{\pi-S_{r}} \right) \leq 1, \, \text{ for every } r \leq R \,.
\end{equation} 
Clearly $\Lambda(R) \geq 16/9$.\\

To tackle this question we take a general point of view, considering arbitrary functions depending on $f$ and $r$. Given $\varphi \colon \mathcal{B} \times [0, 1/3] \to [0,\infty)$, we look for constants $\lambda_{R}$ so that for each $f \in \mathcal{B} $ we have
\begin{equation} \label{biber}
  \sum_{n=0}^{\infty}|a_{n}|r^n + \lambda_{R} \varphi (f, r)\leq 1, 
\end{equation}
for every $0 \leq r \leq R$. With this in mind, the \textit{sharp constant} of $\varphi $ for $0 \leq R \leq 1/3$ is defined as
\[
\Lambda_{\varphi}(R) \coloneqq \sup\left\{\lambda_{R}\geq 0: \sum_{n=0}^{\infty} |a_n|r^{n}+\lambda_{R}\varphi(f,r)\leq 1, \,\text{ for all } f \in \mathcal{B} \text{ and } 0 \leq r \leq R \right\} \,.
\]
Whenever there is no risk of confusion we will simply write  $\Lambda(R)$. So, our main aim from now on is to give, for a given function $\varphi$, a reasonable description of $\Lambda_{\varphi}(R)$. \\

In the definition (and other characterisations, see Proposition~\ref{th:ch} below), all functions in $\mathcal{B} $ are taken into account. We will see that, under certain mild conditions on $\varphi $ (see Theorem~\ref{th:ch_mobius}) it is enough to consider functions of the form 
\begin{equation} \label{moebius}
\phi_{a}(z)=\frac{z-a}{1-az}
\end{equation}
for $ 0\leq  a \leq  1$ (note that, for $0\leq a <1 $ these are Möbius transformations and $\phi_{1}$ is the constant function $-1 $). This makes the problem far easier to handle, and in certain cases will allow us to give an explicit formula for $\Lambda_{\varphi}(R) $.\\

Before we proceed, let us note that, defining $\Upsilon \colon \mathcal{B} \times [0, 1/3] \to [0,\infty)$
\begin{equation}\label{upsilon}
\Upsilon(f,r)\coloneqq 1-\sum_{n=0}^{\infty}|a_{n}|r^n
\end{equation}
the sharp constant of a function $\varphi $ can be rewritten as
\begin{equation} \label{pogacar}
\Lambda(R)
=\sup\left\{\lambda_{R}\geq 0: \lambda_{R}\leq \frac{\Upsilon(f,r)}{\varphi(f,r)}, \,\text{ for all } f \in \mathcal{B} \text{ and } 0 \leq r \leq R \right\} \,.
\end{equation}

\section{Sharp constants}

As a straightforward consequence of the reformulation of $\Lambda (R) $ given in \eqref{pogacar} we have the following result, which combined with Proposition~\ref{prop:feasible_constant} gives a way to check if a given function is feasible.

\begin{proposition} \label{th:ch}
The sharp constant of a function $ \varphi$ is given by
\begin{equation*}
\Lambda(R)=\inf_{\mathcal{B}\times[0,R]}\frac{\Upsilon(f,r)}{\varphi(f,r)} \,.
\end{equation*}
\end{proposition}

Note that the supremum in the definition of $\Lambda(R) $ is actually a maximum. In other words, 
\begin{equation*}
    \sum_{n=0}^{\infty} \modulus{a_n}r^{n}+\Lambda(R)\varphi(f,r)\leq 1 \,,
\end{equation*}
for every holomorphic function $f$ with $\vert f(z)  \vert \leq 1$ for all $z \in \mathbb{D} $ and every $0\leq r \leq R $.

This expression is difficult to compute because it involves an infimum over $\mathcal{B} $ (which in the end is the unit ball of the Banach space $H_{\infty} (\mathbb{D}) $ of bounded holomorphic functions on the disk). Our next step is to show that, under certain conditions on $\varphi $, it is enough to consider the functions $\phi_{a} $ from \eqref{moebius}. This gives an expression that will be much easier to deal with. \\

Let us note that simple computations yield
\begin{equation} \label{manzano}
\phi_{a}(z)=-a+\sum_{n=1}^{\infty}(1-a^2)a^{n-1}z^n
\end{equation}
for each $0\leq a \leq 1 $. This allows us to give a precise formula for $\Upsilon(\phi_{a},r) $: if $0\leq a \leq 1 $ and $0 \leq r \leq 1/3 $, then
\begin{equation} \label{murdoch}
\begin{split}
\Upsilon(\phi_{a},r)
=1-\sum_{n=0}^{\infty} \modulus{a_n}r^{n}=(1-c)-& \sum_{n=1}^{\infty}a_n r^{n}=(1-a)-(\varphi_{a}(r)+c) \\
& =1-2a-\frac{r-a}{1-ar}=1-a-r \frac{1-a^2}{1-ar} \,.
\end{split}
\end{equation}
On the other hand, from \cite[Lemma~3]{second} we know that if $f \in \mathcal{B}$, then
\begin{equation} \label{blueraincoat}
\sum_{k=1}^{\infty}\modulus{a_{n}}r^{n}\leq 
\left\{ 
\begin{aligned}
r \dfrac{1-\modulus{a_{0}}^{2}}{1-r\modulus{a_{0}}} & \text{ if } \modulus{a_{0}}\geq r,\\ 
r \dfrac{\sqrt{1-\modulus{a_{0}}^{2}}}{\sqrt{1-r}} & \text{ if } \modulus{a_{0}}< r.
\end{aligned}\right.
\end{equation}
(just take $p=1 $ and $m=0$ in the statement).

\begin{theorem}\label{th:ch_mobius}
Let $ \varphi$ be a function such that the following conditions hold
\begin{enumerate}[(i)]
\item \label{c:1} $\varphi(f,r)\leq \varphi(\phi_{\vert a_{0} \vert },r)$  for each $0 \leq r \leq 1/3$ and for every $f \in   \mathcal{B}$
\item \label{c:2} $\displaystyle 
    \inf_{\genfrac{}{}{0pt}{}{0 \leq a \leq 1}{0 \leq s \leq R}} \frac{\Upsilon(\phi_{a},s)}{\varphi(\phi_{a},s)}
    \leq   \inf_{0 \leq s \leq R}\Bigg(  \inf_{0 \leq a \leq s} \bigg( 1-a-s\frac{\sqrt{1-a^2}}{\sqrt{1-s^2}} \bigg) \frac{1}{\varphi(\phi_{a},s)} \Bigg) $.
\end{enumerate}
Then, its sharp constant is given by
\begin{equation*}
\Lambda(R)=  \inf_{\genfrac{}{}{0pt}{}{0 \leq a \leq 1}{0 \leq r \leq R}} \frac{\Upsilon(\phi_{a},r)}{\varphi(\phi_{a},r)}.
\end{equation*}
\end{theorem}

\begin{proof}
Fix $ 0\leq R\leq 1/3$. From Theorem~\ref{th:ch} we have
\begin{equation*}
    \Lambda(R)=\inf_{\genfrac{}{}{0pt}{}{f \in \mathcal{B}}{0 \leq r \leq R}}\frac{\Upsilon(f,r)}{\varphi(f,r)}\leq \inf_{\genfrac{}{}{0pt}{}{0 \leq a \leq 1}{0 \leq r \leq R}}\frac{\Upsilon(\phi_{a},r)}{\varphi(\phi_{a},r)} \,.
\end{equation*}
Therefore it is only left to prove that the converse inequality holds, and to do that is suffices to show that
\begin{equation} \label{endrivell}
    \sum_{n=0}^{\infty} \modulus{a_n}r^{n}+ \bigg( \inf_{\genfrac{}{}{0pt}{}{0 \leq a \leq 1}{0 \leq s\leq R}}\frac{\Upsilon(\phi_{a},s)}{\varphi(\phi_{a},s)} \bigg) \varphi(f,r)\leq 1 \,,
\end{equation}
for every $ f \in \mathcal{B} $ and $0 \leq r \leq R $. We fix, then, some $f$ and $r $, and distinguish two cases. First of all, if $r \leq \modulus{a_{0}} \leq 1 $ we know from \eqref{blueraincoat} that
\[
\sum_{k=1}^{\infty}\modulus{a_{n}}r^{n}\leq r \dfrac{1-\modulus{a_{0}}^{2}}{1-r\modulus{a_{0}}} \,.
\]
On the other hand, \eqref{murdoch} and condition \ref{c:1} give
\begin{multline*}
 \inf_{\genfrac{}{}{0pt}{}{0 \leq a \leq 1}{0 \leq s\leq R}}\frac{\Upsilon(\phi_{a},s)}{\varphi(\phi_{a},s)}
 \leq \frac{\Upsilon(\phi_{\modulus{a_{0}}},s)}{\varphi(\phi_{\modulus{a_{0}}},s)}\\
 = \bigg( 1 - \modulus{a_{0}} - r \frac{1-\modulus{a_{0}}^{2}}{1-r\modulus{a_{0}}} \bigg) \frac{1}{\varphi(\phi_{\vert a_{0}\vert},r)}
 \leq   \bigg( 1 - \modulus{a_{0}} - r \frac{1-\modulus{a_{0}}^{2}}{1-r\modulus{a_{0}}} \bigg) \frac{1}{\varphi(f,r)} \,.
\end{multline*}
These two facts immediately yield \eqref{endrivell}.\\
Suppose now that $0 \leq \vert a_{0} \vert \leq r $. Condition \ref{c:2} and \eqref{murdoch} give
\begin{multline*}
\inf_{\genfrac{}{}{0pt}{}{0 \leq a \leq 1}{0 \leq s\leq R}}\frac{\Upsilon(\phi_{a},s)}{\varphi(\phi_{a},s)}
\leq    \inf_{0 \leq s \leq R}\Bigg(  \inf_{0 \leq a \leq s} \bigg( 1-a-s\frac{\sqrt{1-a^2}}{\sqrt{1-s^2}} \bigg) \frac{1}{\varphi(\phi_{a},s)} \Bigg) \\
\leq \inf_{0 \leq a \leq r} \bigg( 1-a-r\frac{\sqrt{1-a^2}}{\sqrt{1-r^2}} \bigg) \frac{1}{\varphi(\phi_{a},s)} 
\leq \bigg( 1-\vert a_{0} \vert-r\frac{\sqrt{1-\vert a_{0} \vert^2}}{\sqrt{1-r^2}} \bigg) \frac{1}{\varphi(f,r)} \,.
\end{multline*}
This, combined with condition  \ref{c:1} and \eqref{blueraincoat} show that \eqref{endrivell} holds also in this case.
\end{proof}

The factor in Theorem~\ref{th:ch_mobius}--\ref{c:2} is going to play some rôle in what follows. To handle it, we define the function $J: [0,1] \times [0, 1/3] \to \mathbb{R} $ given by
\[
J(a,r)\coloneqq 1-a-r\frac{\sqrt{1-a^2}}{\sqrt{1-r^2}} \,.
\]

\begin{remark}\label{th:sufficient} 
If, for each fixed $0\leq r \leq 1/3 $ the function given by $a \rightsquigarrow \frac{J(a,r)}{\varphi(\phi_{a}, r)} $ is decreasing in $[0,r] $, then condition \ref{c:2} in Theorem~\ref{th:ch_mobius} is trivially satisfied.\\
Let us note that the function $J_{r}(a) \coloneqq J(a,r) $ is positive and decreasing in $[0,r] $ for each $r$. Indeed, we have 
\begin{equation*}
J_{r}'(a)=-1+\sqrt{\frac{r^2}{1-r^2}} \frac{a}{\sqrt{1-a^2}} \,,
\end{equation*}
Since $ a\in [0,r]$ and $ r\in [0,1/3]$, we have
\begin{equation*}
J'(a,r)\leq -1+ \frac{r^2}{1-r^2}\leq -1+\frac{1}{3^2}<0 \,
\end{equation*}
and $J_{r} $ is decreasing. Therefore $ J(a,r)\geq J(r,r)=1-2r>0$.\\
\end{remark}

\begin{remark} \label{co:chmobiusinc}
If, on the other hand,  $\varphi $ is an increasing function with respect to $r$ (for each fixed function $f \in \mathcal{B} $), then things get easier, and its sharp constant is given by 
\begin{align*}
    \Lambda(R)&= \sup\left\{\lambda_{R}\geq 0: \sum_{n=0}^{\infty} |a_n|R^{n}+\lambda_{R}\varphi(f,R)\leq 1, \, \text{ for all } f\in \mathcal{B}\right\}\\
    &=\sup\left\{\lambda_{R}\geq 0: \lambda_{R}\leq \frac{\Upsilon(f,R)}{\varphi(f,R)}, \, \text{ for all } f\in \mathcal{B}\right\}\\
    & =\inf_{f \in \mathcal{B}} \frac{\Upsilon(f,R)}{\varphi(f,R)} \,,
\end{align*}
(recall Proposition~\ref{th:ch} for the last formulation). If $\varphi $ moreover satisfies conditions \ref{c:1} and \ref{c:2} in Theorem~\ref{th:ch_mobius}, then 
\begin{equation*}
\Lambda(R)=\inf_{0 \leq a \leq 1} \frac{\Upsilon(\phi_{a},R)}{\varphi(\phi_{a},R)}.
\end{equation*}
\end{remark}

\bigskip

We say that a function $ \varphi:\mathcal{B}\times[0,1/3]\to [0,\infty)$ is \textit{feasible} if there exists some constant $ \lambda>0$ such that for every holomorphic $ f(z)= \sum_{n=0}^{\infty}a_{n}z^n$ with $\vert f(z) \vert \leq 1 $ we have,
 \begin{equation}\label{inequality} 
     \sum_{n=0}^{\infty}|a_{n}|r^n + \lambda \varphi(f,r)\leq 1, \, \text{ for every } r \leq \frac{1}{3}.
 \end{equation}
 With this notation, Theorem~\ref{th:second} gives that the function $\varphi(f,r) = \frac{S_{r}(f)}{\pi - S_{r}(f)} $ is feasible. On the other hand, the function $\Upsilon$ defined in \eqref{upsilon} is clearly feasible. Observe also that a function $\varphi $ is feasible if and only if there is $\lambda >0 $ such that 
 \begin{equation} \label{devienne}
 \lambda \varphi (f,r)  \leq \Upsilon (f,r)
 \end{equation}
 for all $f$ and $r $. This and \eqref{pogacar} give the following characterisation of feasible functions.
 \begin{proposition}\label{prop:feasible_constant} 
 A function $ \varphi$ is feasible if and only if $ \Lambda(1/3)>0$.
 \end{proposition}
 
 We give now some necessary conditions for  feasible functions.

 \begin{proposition} \label{prop:necessary} 
 Let $ \varphi$ be a feasible function and fix $0 \leq r \leq 1/3 $. Then, 
 \begin{enumerate}[(i)]
 \item\label{neces1} $ \varphi(e^{i \theta},r)=0$ for every $\theta \in \mathbb{R}$\,.
 \item\label{neces2} $ \lim_{\modulus{a_{0}} \to 1} \varphi(f,r)=0$\,.
 \item\label{neces3} $ \lim_{a \to 1} \varphi(\phi_{a},r)=0$\,.
 \end{enumerate}
 \end{proposition}
 \begin{proof}
 If $e^{i \theta} $ denotes the constant function, we clearly have $\Upsilon (e^{i \theta},r) =0$, and \eqref{devienne} gives \ref{neces1}.\\
 Fix now $\varepsilon >0 $ and take $f \in \mathcal{B} $ so that $1 - \vert a_{0} \vert < \varepsilon $. We know from \cite{bohr} (see also \cite[Lemma~8.4]{libro}) that  $ 0\leq \modulus{a_{n}}\leq 1-\modulus{a_{0}}^2$ for every $ n\geq 1$. A straightforward computation shows that $ 0\leq \modulus{a_{n}}\leq 2 \varepsilon$  and, then
 \begin{equation*}
     \Upsilon(f,r)= 1-\sum_{n=0}^{\infty} \modulus{a_n}r^{n}=(1-\modulus{a_{0}})-\sum_{n=1}^{\infty} \modulus{a_n}r^{n} < \varepsilon \Big( 1 - \frac{r}{1-r} \Big)\,.
 \end{equation*}
 This shows that $ \lim_{\modulus{a_{0}} \to 1}\Upsilon(f,r)=0 $ and again   \eqref{devienne} gives \ref{neces2}. Finally, \ref{neces3} is a particular case of \ref{neces2}.
 \end{proof}
 
 Note that this implies that no feasible function can be increasing in $ [0,1]$ (recall Remark~\ref{th:sufficient}).

\section{Analysing Bohr's inequality}

We can now address our first goal (stated in \eqref{pregunta}) of making a detailed analysis of Theorem~\ref{th:second}. In this case we have the feasible function
\[
\varphi_{0} (f, r)  =\frac{S_{r}}{\pi-S_{r}} = \frac{1}{\frac{1}{S_{r}/\pi}-1} \,,
\]
where $S_{r} $ was defined in \eqref{Sr}. Note in first place that, for each $0 \leq a \leq 1 $ and $0 \leq r \leq 1/3 $, \eqref{manzano} gives
\begin{equation*}
\frac{S_{r}(\phi_{a})}{\pi}= \sum_{n=1}^{\infty}n \modulus{a_{n}}^{2}r^{2n}=(1-a^{2})^{2}\sum_{n=1}^{\infty}n (a^{n-1})^{2}r^{2n}=\frac{r^{2}(1-a^{2})^{2}}{(1-a^{2}r^{2})^{2}}.
\end{equation*}
Then, 
\begin{equation*}
\frac{1}{S_{r}/\pi}-1
=\frac{(1-r^{2})(1-r^2 a^4)}{r^{2}(1- a^2)^2} \,,
\end{equation*}
and
\begin{equation} \label{follia}
\varphi_{0}(\phi_{a},r)=\frac{r^2}{1-r^2}\cdot \frac{(1-a^2)^2}{1-a^4r^2}\,.
\end{equation}
Also, from \eqref{murdoch} we have
\[ 
\Upsilon(\phi_{a},r)=\frac{1-a}{1-ar}\left( 1-r(1+2a) \right) \,,
\]
for every $a$ and $r $.

\begin{theorem}\label{th:second_improved}
For every fixed $ 0\leq R\leq 1/3$, and every holomorphic function $ f(z)= \sum_{n=0}^{\infty}a_{n}z^n$ such that $\vert f(z) \vert \leq 1 $ for every $z \in \mathbb{D} $ we have
\begin{equation*}
\sum_{n=0}^{\infty} \modulus{a_n}r^{n}+\Lambda(R)\left( \frac{S_{r}}{\pi-S_{r}}\right)\leq 1, \, \text{ for every } 0 \leq r \leq R\,,
\end{equation*}
where the sharp constant is given by
\begin{equation}\label{eq:constantexp} 
    \Lambda(R)=\inf_{a\in [0,1]}\frac{1-R^2}{R^2}\cdot \frac{1-a^4R^2}{(1-a)(1+a)^2}\cdot \frac{1-R(1+2a)}{1-aR}.
\end{equation}
\end{theorem}
\begin{proof}
The result follows directly from Theorem~\ref{th:ch_mobius} (actually from Remark~\ref{co:chmobiusinc}, since $ \varphi_{0}(f,\cdot)$ is clearly increasing for each fixed $ f$), once we have checked that the function $\varphi_{0} $ satisfies the two conditions in the statement. To begin with, fix some $f \in \mathcal{B} $ and note that \cite[Lemma~1]{first} yields
\begin{equation*}
    \frac{S_{r}}{\pi}\leq \frac{r^{2} (1-\modulus{a_{0}}^2)^{2}}{(1-\modulus{a_{0}}^2r^{2})^{2}} \,,
\end{equation*} 
for every $  0<r\leq \frac{1}{\sqrt{2}}$, which gives
\begin{equation*}
    1-\frac{S_{r}}{\pi}\geq \frac{(1-r^2)(1-r^{2}\modulus{a_{0}}^4)^{2}}{(1-\modulus{a_{0}}^2r^{2})^{2}} \,.
\end{equation*} 
Hence (recall \eqref{follia})
\begin{equation} \label{bruckner}
\varphi_{0}(f,r)\leq \frac{r^{2}(1-\modulus{a_{0}}^2)^2}{(1-r^{2})(1-r^2\modulus{a_{0}}^4)} = \varphi_{0} (\phi_{\vert a_{0} \vert},r )
\end{equation}
for every $f \in \mathcal{B} $ and $1 \leq r \leq 1/3 $, showing that \ref{c:1} in Theorem~\ref{th:ch_mobius} holds.\\
We show that  condition \ref{c:2} is satisfied by using Remark~\ref{th:sufficient}. So, it is enough to see that, for each fixed $0 \leq r \leq 1/3 $, the function
\begin{equation} \label{oconnor}
a \rightsquigarrow   \frac{J(a,r)}{\varphi_{0}(\phi_{a},r)}=\frac{1-r^2}{r^2}\left( 1-a-r\sqrt{\frac{1-a^2}{1-r^2}}\right) \frac{1-a^4r^2}{(1-a^2)^2}
\end{equation}
is decreasing and positive in $[0,r] $. We start by fixing some such $r$ and rewriting the quotient as
\begin{equation} \label{haendel}
\frac{J(a,r)}{\varphi_{0}(\phi_{a},r)}=\frac{1-r^2}{r^2} \overbrace{(1-a^4r^2)}^{P_{1}(a)} \overbrace{\frac{1}{(1-a)(1+a)^2}}^{P_{2}(a)} \overbrace{\left(1-\sqrt{\frac{r^2}{1-r^2}}\sqrt{\frac{1+a}{1-a}}\right)}^{P_{3}(a)} \,.
\end{equation}
Since $ \frac{1-r^2}{r^2}>0$, we only have to check that each one of the factors $P_{i} $ for $i=1,2,3 $ is decreasing (as functions of $a$) in $[0,r] $ and, to see that they are positive, that evaluated at $a=r$ are positive. Both conditions are clearly satisfied for $P_{1} $. To see that $P_{2} $ is decreasing it is enough to check that the polynomial $(1-a)(1+a)^{2} $ is increasing in $[0,r] $. An easy exercise with the first and second derivatives gives that it has a local minimum at $a=-1 $, and a local maximum at $a=1/3 $; hence it is increasing in $[0,r] $, proving our claim. Again, it is clear that $P_{2}(r)>0 $. Finally, $P_{3} $ is clearly decreasing in $[0,r] $. As to the evaluation at $a=r $, note that
\begin{equation*}
    \sqrt{\frac{r^2}{1-r^2}}\sqrt{\frac{1+r}{1-r}}=\sqrt{\frac{r^2}{(1-r)^2}}=\sqrt{\frac{r^2}{r^2-2r+1}}<1 \,,
\end{equation*}
since $ 2r<1$. This gives $P_{3}(r)>0 $ and completes the proof.
\end{proof}

\begin{remark}
Let us study now the behaviour of the sharp constant in Theorem~\ref{th:second_improved}. To do that we consider the function
\[
  M(a,r)\coloneqq \frac{\Upsilon(\phi_{a},r)}{\varphi_{0}(\phi_{a},r)}=\frac{1-r^2}{r^2}\cdot \frac{1-a^4r^2}{(1-a)(1+a)^2}\cdot \frac{1-r(1+2a)}{1-ar} \,,
\]
and we have that $\Lambda(R) = \inf_{0 \leq a\leq1} M(a,R) $ for each $0 \leq R \leq 1/3 $. We note in first place that, taking $R=1/3 $ we get
\begin{equation*}
\Lambda(1/3) =   \inf_{0\leq a<1} M(a,1/3)=\inf_{0\leq a<1} \frac{16}{9} \frac{9-a^4}{(1+a)^2(3-a)}=\frac{16}{9} \,,
\end{equation*}
recovering in this way \cite[Theorem~10]{second}. \\
We see first that the infimum in \eqref{eq:constantexp} is actually a minimum. Note that, for each fixed $ 0 \leq r \leq 1/3$, the function $f_{r}(a)\coloneqq M(a,r) $ is clearly continuous in $[0,1) $ and satisfies $ \lim_{a\to 1^{-} } f_{r}(a) = + \infty$. Thus there is  $ a^{*}(r) \in [0,1)$ so that 
\[
    \inf_{0\leq a<1} M(a,r)=M(a^*(r),r) \,.
\] 
Elementary computations show that $ f'_{r}(0)= \frac{r^4-1}{r^2} < 0$. Hence the function is decreasing at $ 0$, which gives that, in fact $ a^{*}(r) \in (0,1)$ and it is a local minimum (so, a critical point of $ f_{r}$). Again, elementary (but tedious) computations show that the critical points of $f_{r}$ in $ (0,1)$ are the roots of the polynomial
\begin{multline*}
p_{r}(x) \coloneqq 2 r^4 x^7 +2  (r-2) r^3 x^6+ r^{2} \left(-7 r^2-4 r+1\right) x^5+ r^{2} \left(-3 r^2+12 r+1\right) x^4\\
+2  r^2 (2 r+1) x^3 +2 r (r-4) x^2- \left(r^2-3\right) x -r^2-1 \,.
\end{multline*}
Note that, since $ p_{0}(x)= 3x-1$, we have $ \lim_{r \to 0} a^{*}(r)=\frac{1}{3}$.\\
Let us give now lower and upper bounds for these sharp constants´. For $ 0 \leq a \leq 1$, $ 0 \leq r \leq 1/3$ and $ \lambda \geq 0$ consider the functions
\begin{align*}
C_{1}(a,r, \lambda )& \coloneqq a+r \frac{1-a^2}{1-ar}+\lambda \frac{r^2}{1-r^2} \frac{(1-a^2)^2}{1-a^4 r^2} \, ,\\
C_{2}(a,r,\lambda) &\coloneqq a+r \sqrt{\frac{1-a^2}{1-r^{2}}}+\lambda \frac{r^2}{1-r^2} \frac{(1-a^2)^2}{1-a^4r^2}\,.
\end{align*}
Then , as a consequence of  \eqref{blueraincoat} we have 
\begin{equation*}
    \sum_{n=0}^{\infty} \modulus{a_n}r^{n}+\lambda \left( \frac{S_{r}}{\pi-S_{r}}\right) \leq 
    \left\{
    \begin{aligned}
    C_{1}(\vert a_{0} \vert,r,\lambda), \text{ if }  0 \leq \vert a_{0} \vert \leq r,\\
    C_{2}(\vert a_{0} \vert,r,\lambda), \text{ if }  r \leq \vert a_{0} \vert \leq 1 \,.
    \end{aligned}\right.
\end{equation*}
Each one of these is increasing on $ r$ (for fixed $ a$ and $ \lambda$), and we know (see again \cite{second}) that $ C_{i} (a, 1/3 , 16/9) \leq 1$ for every $0 \leq  a \leq 1$ and $0 \leq  r \leq 1/3$ (and $ i=1,2$). Consider $ \lambda_{0}(r) = \frac{2(1-r^2)}{9r^{2}}$ and note that, for each $ 0 \leq r \leq 1/3$ we have
\[
C_{i} (a, r, \lambda_{0}(r)) \leq C_{i} (a, 1/3, \lambda_{0}(r))  = C_{i}(a, 1/3, 16/9 ) \leq 1 \,.
\]
This shows that \eqref{pregunta} holds for $ \lambda_{0}(r)$ and then, by the definition of sharp constant
\[
\frac{2}{9}\left(\frac{1-R^2}{R^2}\right) \leq \Lambda(R) \,.
\]
For an upper bound let us simply observe that
\begin{multline*}
\inf_{0\leq a<1} M(a,R)\leq M(1/3,R)
= \frac{27}{32} \left(\frac{1-R^2}{R^2}\right)\frac{(1-\frac{R^2}{81})(1-\frac{5R}{3})}{1-\frac{R}{3}} \\
= \frac{1}{96}   \left(\frac{1-R^2}{R^2}\right) \left( \frac{(81-R^{2})(3-5R)}{(3-R)} \right)
\end{multline*}
Taking all these into account we have
\[ 
\frac{2}{9}\left(\frac{1-R^2}{R^2}\right) \leq \Lambda(R) \leq 
 \frac{1}{96}  \left( \frac{(81-R^{2})(3-5R)}{(3-R)} \right) \left(\frac{1-R^2}{R^2}\right)  \,.
\]
To give a better idea, we include here the graphs of the functions defining the lower and upper bounds, for the intervals $ [0,1/3] $ and $ [1/4,1/3]$ respectively.\\

\begin{center}
\begin{tabular}{cc}
\includegraphics[width=0.4\textwidth]{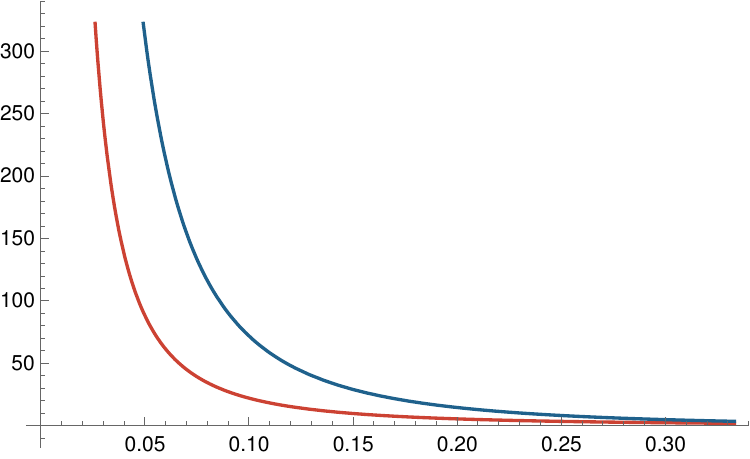}
 &
\includegraphics[width=0.4\textwidth]{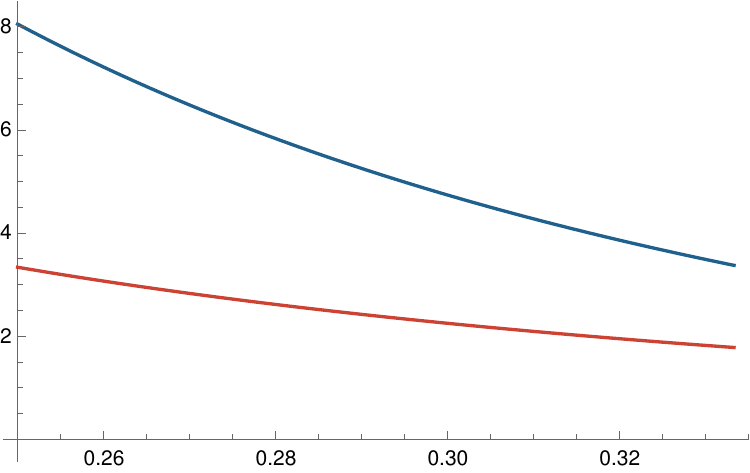}
\end{tabular}
\end{center}
\end{remark}

\begin{remark} \label{th:second_improved_g}
If $ g$ is a bounded, positive function on $ [0,1]$ we may consider the function
\[
\varphi (f,r) = g\big(\vert f(0) \vert\big) \varphi_{0} (f,r)  = g\big(\vert a_{0} \vert\big) \frac{S_{r}}{\pi - S_{r}}
\]
for $ f(z) = \sum_{n=0}^{\infty} a_{n} z^{n}$ in $ \mathcal{B}$ and $ 0<r<1/3$. Since $ g$  is positive,  \eqref{bruckner} yields
\[
\varphi (f, r) \leq \varphi (\phi_{\vert a_{0} \vert} ,r)
\]
for every $ f$  and $ r$. Now, if $ g$ is such that, for each fixed $0\leq r \leq 1/3 $ the function given by 
\begin{equation} \label{mas}
a \rightsquigarrow \frac{J(a,r)}{g(a) \varphi_{0}(\phi_{a}, r)}
= \frac{1-r^2}{r^2}\cdot \frac{1}{g(a)} \cdot \frac{1-a^4r^2}{(1-a^2)^2} \cdot  \Big( 1-a-r\frac{\sqrt{1-a^2}}{\sqrt{1-r^2}}  \Big)
\end{equation}
is decreasing in $[0,r] $, then we can apply  Theorem~\ref{th:ch_mobius} to have that 
\begin{equation*}
    \sum_{n=0}^{\infty} \modulus{a_n}r^{n}+\Lambda(R)\left( g(\modulus{a_{0}}) \frac{S_{r}}{\pi-S_{r}}\right)\leq 1\,, 
\end{equation*}
 for every $ f \in \mathcal{B}$ all $ 0 \leq R \leq 1/3$ and $ 0 < r \leq R$. Here the sharp constant is given by
\begin{equation} \label{weinenklagen}
    \Lambda(R)
    = \bigg(\frac{1-R^2}{R^2}\bigg) \inf_{0 \leq a \leq 1} \bigg(\frac{1}{g(a)}\cdot \frac{1-a^4R^2}{(1-a)(1+a)^2}\cdot \frac{1-R(1+2a)}{1-aR}\bigg) \,.
\end{equation}
Condition in \eqref{mas} might not be easy to check. However, recall that we already saw in the proof of Theorem~\ref{th:second_improved} that $ a \rightsquigarrow  \frac{J(a,r)}{\varphi_{0}(\phi_{a}, r)}$ is decreasing as a function of $ a$ (see \eqref{oconnor}). With this,  if $ g$  is increasing, then the function in \eqref{mas} is clearly decreasing and we have \eqref{weinenklagen}.
\end{remark}

The challenge now is to find functions satisfying \eqref{mas} that improves Theorem~\ref{th:second}, in the sense that, keeping the factor  $\frac{S_{r}}{\pi- S_{r}}$ we can add a term which for each function is bigger than $ \frac{16}{9}$.

\begin{example}
We take the function  $ g\colon [0,1] \to \mathbb{R}$ given by 
\begin{equation*}
g(a)=\frac{(9+a^4)}{(2+a^{2})(3-a)} \,,
\end{equation*}
which is clearly positive. We need to see that \eqref{mas} is satisfied. To do that let us note that $ (3-a)$ is decreasing as a function of $ a$, and so also is $  \Big( 1-a-r\frac{\sqrt{1-a^2}}{\sqrt{1-r^2}}  \Big) \Big( \frac{2+a^{2}}{9-a^{2}} \Big)$. Taking this  and \eqref{haendel} into account, we have that \eqref{mas} is indeed satisfied. Then by \eqref{weinenklagen} we have
\begin{equation*}
\Lambda(1/3) = \frac{16}{9} \inf_{0<a<1}  \frac{ \left(a^2+2\right) \left(a^4-9\right)}{ (a+1)^2 \left(a^2-9\right)} = \frac{4}{3} \,,
\end{equation*}
and
\begin{equation*}
\sum_{n=0}^{\infty} \modulus{a_n}r^{n}
+ \frac{4}{3} \cdot \frac{(9+ \vert a_{0} \vert^4)}{(2+\vert a_{0} \vert^{2})(3-\vert a_{0} \vert)}
\cdot  \frac{S_{r}}{\pi-S_{r}} \leq 1\,, 
\end{equation*}
for every $ f \in \mathcal{B}$ and all $ 0 \leq r \leq 1/3$. Note finally that $ g(a) \geq g(1) = \frac{4}{3}$ for every $ 0 \leq a \leq 1$. Hence, the factor multiplying $\frac{S_{r}}{\pi-S_{r}}$ is bigger than or equal to $\frac{16}{9}$ for every function, improving in this way Theorem~\ref{th:second}.
\end{example}

Theorem~\ref{th:ch_mobius} provides a way to find functions that produce an improved version of Bohr's inequality. As we already noted, condition \ref{c:2} in Theorem~\ref{th:ch_mobius} is somewhat difficult to handle, so in Remark~\ref{th:sufficient} we gave an easier condition. On the other hand, we look for functions that are as big as possible, so that the extra factor  that we are adding in \eqref{biber} is as big as possible. We give now a step in this direction.  Consider the  function $ \Psi : [0,1] \times [0,1/3] \to \mathbb{R}$ given by
\[ 
\Psi (a,r) \coloneqq \left\{
\begin{array}{ll}
1 - a - r \dfrac{1-a^2}{1-ar} \, , & \text{ if }r \leq a \leq 1 \\ \\
 1-a-r\dfrac{\sqrt{1-a^2}}{\sqrt{1-r^2}}  \, ,& \text{ if } 0 \leq a \leq r 
\end{array}
\right.
\]
and define $ \psi \colon \mathcal{B} \times [0,1/3] \to \mathbb{R}$ by
\[ 
\psi (f, r) \coloneqq \Psi (\vert f(0) \vert , r) \,.
 \]
 This function obviously satisfies both conditions in Theorem~\ref{th:ch_mobius}; hence its sharp constant is $ 1$ and we have the following strengthening of \eqref{blueraincoat}.
 
\begin{theorem} \label{final}
For every holomorphic function $ f(z)= \sum_{n=0}^{\infty}a_{n}z^n$ such that $\vert f(z) \vert \leq 1 $ for every $z \in \mathbb{D} $, and $ 0 \leq r \leq 1/3$ we have
\begin{equation*}
\sum_{n=0}^{\infty} \modulus{a_n}r^{n} + \Big( 1 - \modulus{a_{0}} - r \dfrac{1- \modulus{a_{0}}^2}{1-\modulus{a_{0}}r} \Big) \leq 1 \,, \text{ if } r \leq \vert a_{0} \vert  \leq 1
\end{equation*}
and
\begin{equation*}
\sum_{n=0}^{\infty} \modulus{a_n}r^{n} + \Big(  1-\modulus{a_{0}}-r\dfrac{\sqrt{1-\modulus{a_{0}}^2}}{\sqrt{1-r^2}}  \Big) \leq 1 \,, \text{ if } 0 \leq \vert a_{0} \vert  < r \,.
\end{equation*}
Moreover, the inequalities are optimal, in the sense that there is no $ \lambda_{r} >1$ so that $ \sum_{n=0}^{\infty} \modulus{a_n}r^{n} + \lambda_{r} \psi(f,r) \leq 1$ for every $ f \in \mathcal{B}$.
\end{theorem}

\begin{remark}
Let us briefly compare Theorem~\ref{final} with Theorem~\ref{th:second}, and see that $ \frac{16}{9} \varphi_{0}(f,r) \leq \psi (f, r)$ for every $ f \in \mathcal{B}$ and $ 0 \leq r \leq 1/3$. Fix $ 0 \leq r \leq 1/3$  and pick some $ f \in \mathcal{B}$. We know from the proof of Theorem~\ref{th:second_improved} that $ \varphi_{0} (f, r) \leq \varphi_{0} (\phi_{\vert a_{0} \vert}, r)$. Now we have to consider two cases. Suppose first that $ 0 \leq \vert a_{0} \vert  < r$. We know (recall again the proof of Theorem~\ref{th:second_improved}) that the function $ a \rightsquigarrow \frac{J(a,r)}{\varphi_{0}(\phi_{a}, r)}$ is decreasing in $ [0,r]$ and, then,
\[ 
\frac{J(a,r)}{\varphi_{0}(\phi_{a}, r)} \geq \frac{J(r,r)}{\varphi_{0}(\phi_{r}, r)}
= \frac{(1-2r)(1-r^{6})}{r^{2}(1-r^{2})} =  \frac{(1-2r)}{r^{2}} (1+r^{2}+r^{4}) \geq \frac{91}{27} \,.
\]
This immediately gives 
\[ 
 \Big(  1-\modulus{a_{0}}-r\dfrac{\sqrt{1-\modulus{a_{0}}^2}}{\sqrt{1-r^2}}  \Big) 
 \geq \frac{16}{9} \frac{S_{r}}{\pi - S_{r}}
\]
whenever $ 0 \leq \vert a_{0} \vert  < r$.\\
Suppose now that $ r \leq \vert a_{0} \vert  \leq  1$. In this case we have as a straightforward consequence of Proposition~\ref{th:ch} and the fact that the optimal constant for $ \varphi_{0}$ is $ \Lambda_{\varphi_{0}}(1/3)= \frac{16}{9}$, that
\[ 
 \Big(  1-\modulus{a_{0}}-r\dfrac{1-\modulus{a_{0}}^2}{1-r^2}  \Big) 
 \geq \frac{16}{9} \frac{S_{r}}{\pi - S_{r}} \,.
\]
In order to get a better feeling of what is going on, we show some plots of the functions $ \psi$ (in red) and $ \varphi_{0}$ (in blue), evaluated at Möbius transforms $ \phi_{a}$ (which, as we have seen, play a fundamental r\^ole).\\

\begin{center}
\begin{tabular}{p{.5\textwidth} p{.5\textwidth}}
\includegraphics[width=0.45\textwidth]{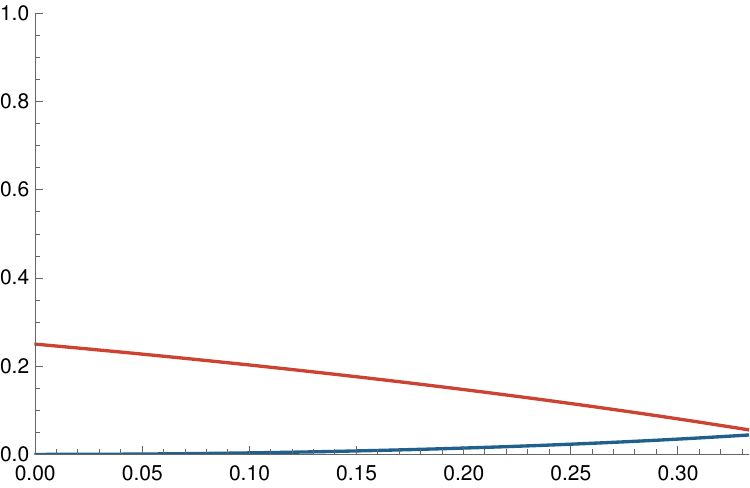}
\newline
{\small $ \psi (\phi_{a}, r)$ and $ \varphi_{0}(\phi_{a},r)$ for $ a=3/4$}
 &
\includegraphics[width=0.45\textwidth]{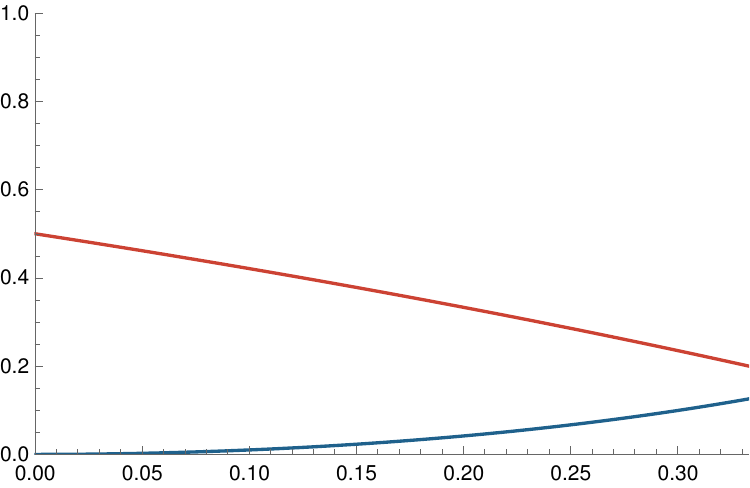}
\newline
{\small $ \psi (\phi_{a}, r)$ and $ \varphi_{0}(\phi_{a},r)$ for $ a=1/2$} \\ & \\

\includegraphics[width=0.45\textwidth]{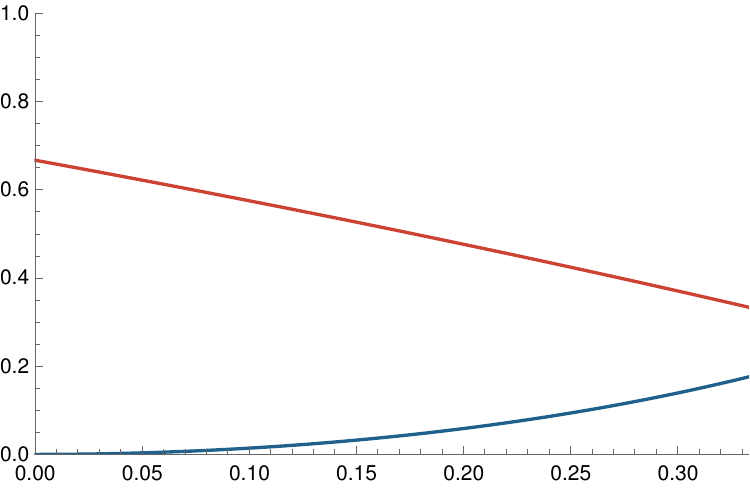}
\newline
{\small $ \psi (\phi_{a}, r)$ and $ \varphi_{0}(\phi_{a},r)$ for $ a=1/3$}
 &
\includegraphics[width=0.45\textwidth]{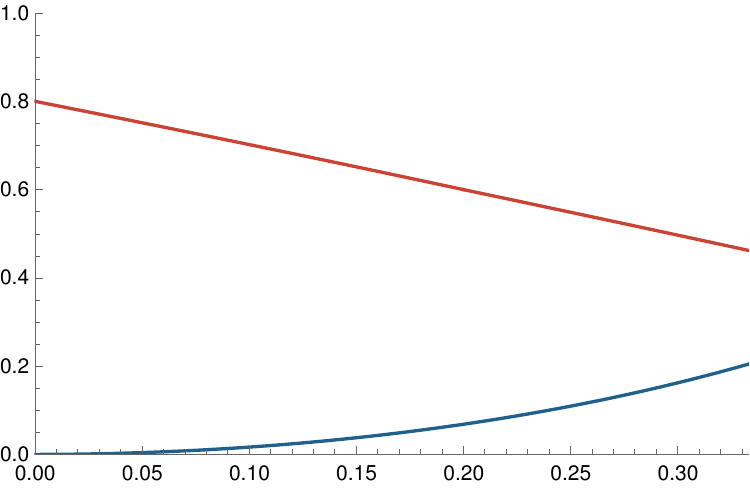}
\newline
{\small $ \psi (\phi_{a}, r)$ and $ \varphi_{0}(\phi_{a},r)$ for $ a=1/5$} \\
\end{tabular}
\end{center}

\begin{center}
\begin{tabular}{p{.5\textwidth} p{.5\textwidth}}
\includegraphics[width=0.45\textwidth]{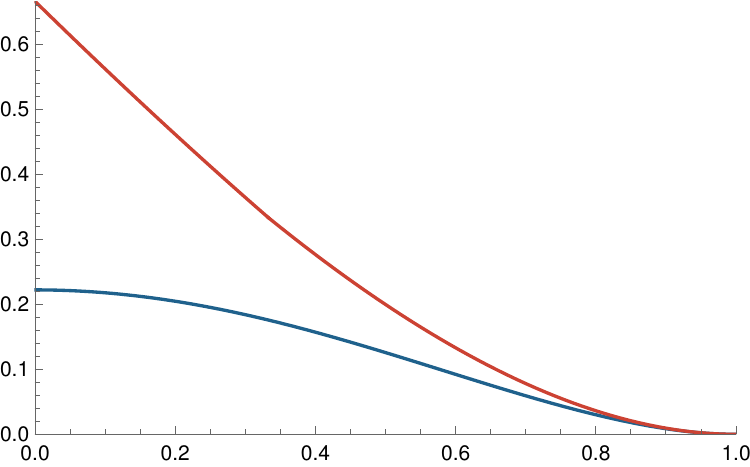}
\newline
{\small $ \psi (\phi_{a}, 1/3)$ and $ \varphi_{0}(\phi_{a},1/3)$}
 &
\includegraphics[width=0.45\textwidth]{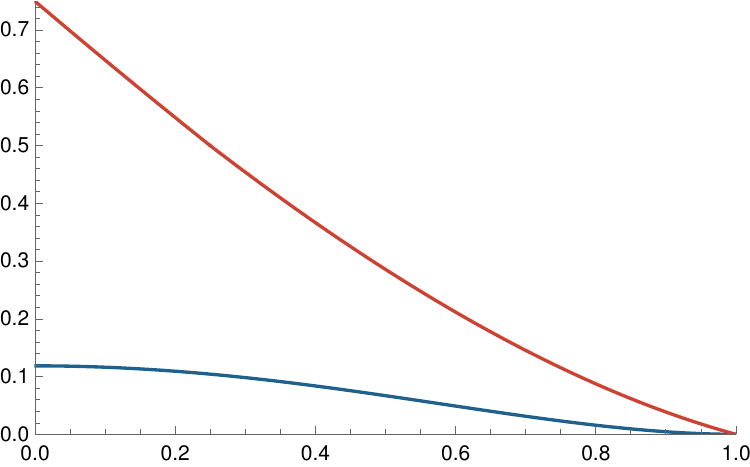}
\newline
{\small $ \psi (\phi_{a}, 1/4)$ and $ \varphi_{0}(\phi_{a},1/4)$} \\ & \\

\includegraphics[width=0.45\textwidth]{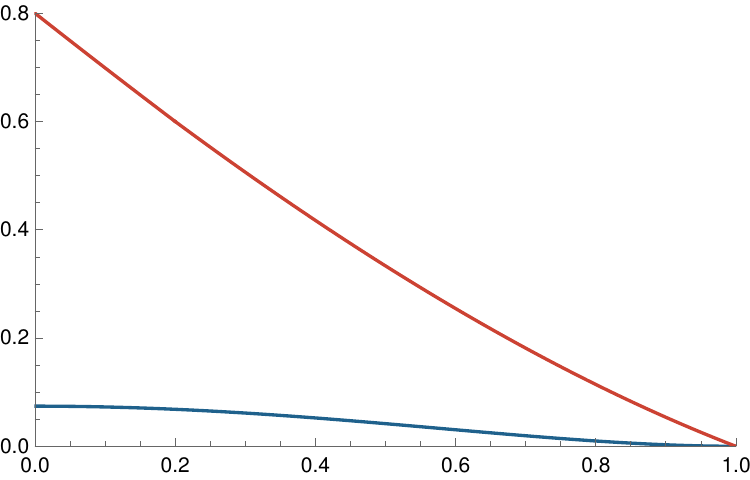}
\newline
{\small $ \psi (\phi_{a}, 1/5)$ and $ \varphi_{0}(\phi_{a},1/5)$}
 &
\includegraphics[width=0.45\textwidth]{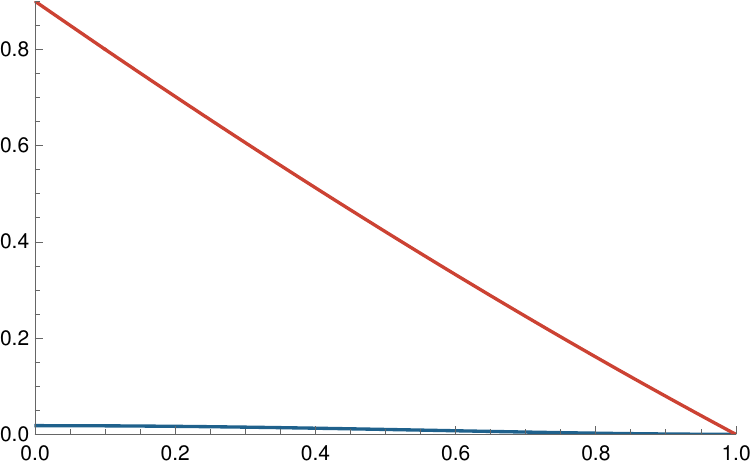}
\newline
{\small $ \psi (\phi_{a}, 1/10)$ and $ \varphi_{0}(\phi_{a},1/10)$} 
\end{tabular}
\end{center}
\end{remark}

\end{document}